\documentclass[12pt,reqno]{amsart}
\usepackage{amscd,amsmath,amsthm,amssymb}
\usepackage{color}
\usepackage{stmaryrd}
\usepackage{graphicx}
\usepackage{leftidx}
\usepackage{tikz}
\usepackage{tikz-cd}
\usepackage[all,cmtip]{xy}

\newtheorem{Theorem}{Theorem}[section]
\newtheorem{Lemma}[Theorem]{Lemma}
\newtheorem{Corollary}[Theorem]{Corollary}
\newtheorem{Proposition}[Theorem]{Proposition}

\newtheorem{Conjecture}[Theorem]{Conjecture}

\newtheorem{Conj}{Conjecture}

\let\epsilon\varepsilon
\def\qed{\ifhmode\textqed\fi
\ifmmode\ifinner\hfill\quad\qedsymbol\else\dispqed\fi\fi}
\def\textqed{\unskip\nobreak\penalty50
\hskip2em\hbox{}\nobreak\hfill\qedsymbol
\parfillskip=0pt \finalhyphendemerits=0}
\def\dispqed{\rlap{\qquad\qedsymbol}}

\def\QQ{\mathbb{Q}}
\def\mm{\mathfrak{m}}
\def\T{\textup{T}}
\def\H{\textup{H}}
\def\Tor{\textup{Tor}}
\def\Ker{\textup{Ker}}
\def\Coker{\textup{Coker}}
\def\Im{\textup{Im}}
\def\height{\textup{height}}
\def\rank{\textup{rank}}
\def\Der{\textup{Der}}
\def\Mod{\textup{Mod}}

\begin{document}
\title{Cotangent functors and\\ Herzog's last theorem}

\dedicatory{Dedicated to the memories of Professors J\"urgen Herzog and Wolmer Vasconcelos}

\author{Antonino Ficarra}

\address{Antonino Ficarra, Departamento de Matem\'{a}tica, Escola de Ci\^{e}ncias e Tecnologia, Centro de Investiga\c{c}\~{a}o, Matem\'{a}tica e Aplica\c{c}\~{o}es, Instituto de Investiga\c{c}\~{a}o e Forma\c{c}\~{a}o Avan\c{c}ada, Universidade de \'{E}vora, Rua Rom\~{a}o Ramalho, 59, P--7000--671 \'{E}vora, Portugal}
\email{antonino.ficarra@uevora.pt}\email{antficarra@unime.it}

\maketitle\vspace*{-0.6cm}

\begin{abstract}
	We present the theory of cotangent functors following the approach of Palamodov, and a conjecture of Herzog relating the vanishing of certain cotangent functors to the property of being a complete intersection.
\end{abstract}

\section*{Preface}

The present article is an attempt to complete Jurgen Herzog's unfinished last paper. I thank Maja, Susanne and Ulrike (Herzog's wife and daughters) for handing me over the manuscript of their father. This paper was being written by Professor Herzog to honor the memory of Professor Wolmer Vasconcelos. It is my intention to honor both Professors Herzog and Vasconcelos, whose works have influenced reaserchers in Commutative Algebra for over the past fifty years.

\tableofcontents\vspace*{-1cm}
\section*{Introduction}

During the 1960s and 1970s, the study of the module of differentials was a central topic in Commutative Algebra, particularly in Germany. The book of Kunz \cite{KuBook} reflects this trend. The module of K\"ahler differentials is defined as a construction satisfying a fundamental universal mapping property, making the corresponding functor representable. The fascination with the subject was partly driven by several famous, and still open conjectures, namely the Zariski-Lipman's conjecture, the Berger's conjecture and the Vasconcelos's conjecture. For a nice summary about what is known about these conjectures we refer the reader to Herzog's survey \cite{H94}. See also the papers \cite{HMM,MM,Mi} for some recent progress on Berger's conjecture.

Another notable and related question again posed by Vasconcelos \cite{V} somehow parallels the previously mentioned conjecture due to him.
\begin{Conj}\label{ConjA}
	Let $R$ be a Noetherian ring and $I\subset R$ be an ideal such that the conormal module $I/I^2$ has finite projective dimension as an $R/I$-module. Then $I$ is a complete intersection.
\end{Conj}
Recently, Benjamin Briggs \cite{BB} resolved Conjecture \ref{ConjA} in full generality.\smallskip

For any $R$-algebra homomorphism $R\rightarrow S$ and any $S$-module $M$, there exist modules $\T^i(S/R,M)$ and $\T_i(S/R,M)$ for $i=0,1,\dots$, the so-called \textit{tangent} and
\textit{cotangent modules}. They are functorial in all three variables. In 1967, Lichtenbaum and Schlessinger \cite{LS} first introduced the functors $\T^i$ for $i=0,1,2$. Later on, Quillen \cite{Qu} in 1970 and Andr\'{e} \cite{An} in 1974 defined the higher cotangent functors and developed their theory. In characteristic 0, a different and simpler approach was given by Palamodov \cite{P} by employing DG-algebras. We present the approach of Palamodov in Section \ref{sec2}.\smallskip

In 1981, J\"urgen Herzog posed the following conjecture, which remains open.
\begin{Conj}\label{ConjB}
	\textup{(Herzog's Last Theorem).} Let $R$ be a regular local ring, $I\subset R$ a proper ideal and $S=R/I$. Then the following conditions are equivalent:
	\begin{enumerate}
		\item[(a)] $I$ is a complete intersection.
		\item[(b)] $\T_i(S/R, S)=0$ for all $i>1$.
		\item[(c)] $\T_i(S/R, S)=0$ for all $i\gg1$.
	\end{enumerate}
\end{Conj}
We would like to refer to this conjecture as \textit{Herzog's Last Theorem}, both as a playful nod to Fermat's Last Theorem, but also because it represents the last mathematical question he studied before his passing. We will explain the relation between Conjecture \ref{ConjA} and Conjecture \ref{ConjB} shortly thereafter.

In this note, we present the theory of cotangent functors $\T_i(S/R, M)$ following the approach of Palamodov \cite{P}. Special attention is reserved for the case when $R$ is a regular local ring, $S=R/I$ is the quotient of $R$ by an ideal $I\subset R$ and $M=S$. Herzog's last theorem is formulated and discussed in this frame. 

The outlines of this note are as follows.

In Section \ref{sec1}, we collect basic material on DG-algebras and the module of K\"alher differentials. In Section \ref{sec2} we present the construction of the cotangent functors due to Palamodov \cite{P}. For the remaining part of the paper, we focus our attention on the following special situation. Let $(R,\mm,K)$ be a Noetherian local ring, $I\subset R$ be a proper ideal, $S=R/I$ and $\varphi: R\rightarrow S$ be the canonical epimorphism.

In Section \ref{sec3} we express the cotangent functors $\T_i(S/R,S)$ in terms of the homology of a certain complex $L$. This complex can be regarded as an approximation of a $S$-resolution of the conormal module $I/I^2$. In view of this observation and Conjecture \ref{ConjA}, Herzog was led to speculate that the vanishing of almost all cotangent functors $\T_i(S/R,S)$ is equivalent to the fact that $I$ is a complete intersection. Partial progress towards Conjecture \ref{ConjB} was made by Herzog himself (Theorem \ref{true}) under the extra assumption that all finitely generated $S$-modules have a rational Poincar\'{e} series. Additional results were contributed by Ulrich \cite{Ul}.

\section{DG-algebras and the module of K\"alher differentials}\label{sec1}

In this first section, we give a brief overview on  DG-algebras and on the module of K\"alher differentials.

\subsection{A glimpse to DG-algebras}

Recall that a \textit{differential graded algebra}, or in short a \textit{DG-algebra}, is graded algebra $X=\bigoplus_{k\ge0}X_k$ which is skew-symmetric together with a differential $\partial$ of degree $-1$, satisfying the Leibniz rule:
\[
\partial(ab)=\partial(a)b+(-1)^{\deg a}a\partial(b)
\]
for all $a,b\in X$ with $a$ homogeneous.

An element $z\in X$ such that $\partial(z)=0$ is called a \textit{cycle}. Given any DG-algebra $Y$ and a cycle $z\in Y_i$, the DG-algebra $Y'=Y[T: \partial(T)=z]$, which is obtained by adjoining $T$ in order to kill the cycle $z$, is defined as follows.

We set $\deg T=i+1$. If $i$ is even, then
\[
Y_j'=Y_j\oplus Y_{j-i-1}T\; \text{
	for all $j$, } \; T^2=0 \text{ and } \partial(T)=z,
\]
and if $i$ is odd, then for all $j$ and $k$,
\[
Y_j'=Y_j\oplus Y_{j-(i+1)}T\oplus Y_{j-2(i+1)}T^2\oplus \cdots, \text{ $T^jT^k=T^{j+k}$ and $\partial(T^j)=jzT^{j-1}$.}
\]
With these definitions, $Y'$ is a well-defined DG-algebra and the inclusion $Y\hookrightarrow Y'$ is a homomorphism of DG-algebras.\smallskip

For further details on DG-algebras we refer the reader to Avramov's lectures \cite{Av}.

\subsection{A glimpse to the module of K\"alher differentials} Let $\varphi:R\rightarrow S$ be a ring homomorphism and let $M$ be a $S$-module. Recall that a \textit{derivation}, or more precisely an \textit{$R$-derivation into} $M$ is a map $D:S\rightarrow M$ satisfying:
\begin{enumerate}
	\item[(i)] (Linearity) $D(rs)=rD(s)$ for all $r\in R$ and all $s\in S$.
	\item[(ii)] (Additivity) $D(s_1+s_2)=D(s_1)+D(s_2)$ for all $s_1,s_2\in S$.
	\item[(iii)] (Leibniz rule) $D(s_1s_2)=s_1D(s_2)+s_2D(s_1)$ for all $s_1,s_2\in S$.
\end{enumerate}

Given two $R$-derivations $D,D'$ into $M$ then $D+D':a\in S\mapsto D(a)+D'(a)\in M$ is an $R$-derivation, and for any $c\in S$ the map $cD:a\in S\rightarrow cD(a)\in M$ is also an $R$-derivation. Thus, the set of all $R$-derivations into $M$ is a $S$-module, which we denote by $\Der_R(S,M)$.

Let $\Mod_S$ be the category of $S$-modules. For any homomorphism of $S$-modules $\alpha:M\rightarrow N$ and any $R$-derivation $D\in\Der_R(S,M)$, then $\alpha\circ D\in\Der_R(S,N)$ is again an $R$-derivation. Hence, the map $M\in\Mod_S\mapsto\Der_R(S,M)$ is a functor. This functor is \textit{representable}. That is, there exist a $S$-module $\Omega_{S/R}$ and a map $d:S\rightarrow\Omega_{S/R}$ such that for any $M\in\Mod_S$ and $D\in\Der_R(S,M)$, there exists a unique homomorphism $\alpha:\Omega_{S/R}\rightarrow M$ making the following diagram commutative
\[
\xymatrix{
	S \ar[r]^{d}\ar[d]_{D} & \Omega_{S/R} \ar@{-->}[dl]^{\alpha} \\
	M}
\]

The $S$-module $\Omega_{S/R}$ is called the \textit{module of} (\textit{K\"ahler}) \textit{differentials}.

\section{The cotangent functors}\label{sec2}

Let $R$ be a commutative ring which contains $\QQ$ as a subring. Assigned to a ring homomorphism $\varphi: R\rightarrow S$ and a $S$-module $M$ there is a sequence of $S$-modules $\T_i(S/R,M)$, for $i\ge0$, satisfying the following properties:\smallskip
\begin{enumerate}
	\item[(i)] $\T_0(S/R,M)\cong \Omega_{S/R}\otimes M$, where $\Omega_{S/R}$ is the module of K\"ahler differentials.
	\item[(ii)] Given short exact sequence of $S$-modules $0\rightarrow M_1\rightarrow M_2\rightarrow M_3\rightarrow 0$, there exists a long exact sequence
	\[
	\cdots \to\T_{i+1}(S/R,M_3)\rightarrow \T_{i}(S/R,M_1)\rightarrow \T_{i}(S/R,M_2)\rightarrow \T_{i}(S/R,M_3)\rightarrow \cdots.
	\]
	\item[(iii)] (The Zariski sequence) For a sequence of ring homomorphisms $R\rightarrow S\rightarrow T$ and a $T$-module $M$, there exists a long exact sequence
	\[
	\cdots \rightarrow \T_{i+1}(T/S,M) \rightarrow \T_{i}(S/R,M)\rightarrow \T_{i}(T/R,M) \rightarrow \T_{i}(T/S,M)\rightarrow \cdots.
	\]
\end{enumerate}\smallskip

We give a brief description of the construction of these functors, following the definition of V.P. Palamodov~\cite{P}. One first constructs a \textit{resolvent}, sometimes also called a \textit{Tate resolution}, $X$ of $\varphi$, see \cite{Tate}. The construction of $X$ is obtained step by step, by adjoining variables using the procedure described in Section \ref{sec1}.

Let $\Phi: P=R[\{Z_{\lambda}\}_{\lambda\in \Lambda}]\rightarrow S$ be a polynomial presentation with $\Phi_{|R}=\varphi$ and $\Phi$ an epimorphism. We construct a free acyclic DG-algebra $X$ over $P$ with $\H_0(X)\cong S$, which we call a resolvent of $\varphi$ as described next.

We regard $P$ as a DG-algebra and set $X^{(0)}=P$. Let $I\subset P$ be the kernel of $\Phi$, and let $\{a_\lambda\}_{\lambda\in \Lambda_1}$ be a system of generators of $I$. We adjoin variables
$\{T_{1\lambda}\}_{\lambda\in \Lambda_1}$ of degree $1$ with $\partial(T_{1\lambda})= a_\lambda$ for all
$\lambda\in \Lambda_1$ to obtain the DG-algebra $X^{(1)}$.

In the next step we adjoin sufficiently many variables of degree $2$ to $X^{(1)}$ in order to kill a set of cycles whose homology classes generate $\H_1(X^{(1)})$, and we denote by $X^{(2)}$ the resulting DG-algebra.

Proceeding in this way we obtain an ascending chain
$$
X^{(0)}\subseteq X^{(1)}\subseteq X^{(2)}\subseteq\cdots
$$
of DG-algebras. Then $X=\lim X^{(i)}$ is a resolvent of $\varphi:R\rightarrow S$.

Note that the resolvent is of the form
\begin{equation}\label{eqP}
	X=P[\{T_{i\lambda}:\lambda\in \Lambda_i, i\ge0 \}]\;\,\, \text{with\, $\deg T_{i\lambda}=i$\, for all\, $\lambda\in \Lambda_i$. }
\end{equation}
It is clear that if $R$ is Noetherian and $S$ is a finitely generated $R$-algebra, then $P$ may be chosen a polynomial ring over $R$ in finitely many variables, and the resolvent may be chosen such that $|\Lambda_i|<\infty$ for all $i$. Then $X$ is a free $P$-module resolution of $S$ with all $X_i$ of finite rank.

Let $X$ be a resolvent of $\varphi$. Recall that the module of \textit{K\"alher differentials} $\Omega_{X/R}$ of $X$ over $R$ is an $X$-module together with a $R$-linear derivation $d: X\rightarrow \Omega_{X/R}$ such that the usual universal properties hold. Since $\Omega_{X/R}$ is an $X$-module, it is a complex and it is required that the differential $\partial: \Omega_{X/R}\rightarrow \Omega_{X/R}$ satisfies the rule $$\partial(av)=\partial(a)v+(-1)^{\deg(a)}a\partial(v)$$ for all homogeneous $a\in X$ and all $v\in \Omega_{X/R}$ and commutes with the derivation, that is, $\partial d=d \partial$. This is always possible. Indeed, we can define the module of differentials $\Omega_{X/R}$ of $X$ over $R$ to be the free $X$-module
\[\
\Omega_{X/R} =\bigoplus_{\lambda\in \Lambda}XdZ_{\lambda}\oplus\bigoplus_{\lambda\in \Lambda_1}XdT_{1\lambda}
\oplus \cdots,
\]
the map $\partial:\Omega_{X/R}\rightarrow\Omega_{X/R}$ to be defined by the formula $\partial(d(T))=d(\partial(T))$, and the map $d:X\rightarrow\Omega_{X/R}$ is uniquely determined by the requirement that it is $R$-linear and that $d(ab)=d(a)b+(-1)^{\deg a}ad(b)$ holds for all $a,b\in X$ with $a$ homogeneous.\smallskip
\begin{Theorem}
	\label{tidef}
	Let $\varphi: R\rightarrow S$ be a ring homomorphism, $M$ a $S$-module and $X$ a resolvent of $\varphi$. Then the homology of the complex $\Omega_{X/R}\otimes_XM$ is independent of the choices made in the construction of $X$.
\end{Theorem}
For the proof of this theorem it is essential that $\QQ\subset R$. A complete proof of this result was given by Bingener and Kosarev, in the Kapitel III of the book \cite{BK}.

Theorem~\ref{tidef} allows us to define
\[
\T_i(S/R,M)=\H_i(\Omega_{X/R}\otimes_XM)
\]
for $i\ge0$, where $X$ is any resolvent of $\varphi$. Here we used that $M$ is an $X$-module via the augmentation map $X\rightarrow \H_0(X)=S$.\smallskip

We sketch the proofs of the properties (i), (ii) and (iii).
\begin{Proposition}
	Let $\varphi:R\rightarrow S$ be a ring homomorphism and let $M$ be a $S$-module. The following statements hold.
	\begin{enumerate}
		\item[\textup{(i)}] $\T_0(S/R,M)\cong \Omega_{S/R}\otimes M$.
		\item[\textup{(ii)}] Given short exact sequence of $S$-modules $0\rightarrow M_1\rightarrow M_2\rightarrow M_3\rightarrow 0$, there exists a long exact sequence
		\[
		\cdots \to\T_{i+1}(S/R,M_3)\rightarrow \T_{i}(S/R,M_1)\rightarrow \T_{i}(S/R,M_2)\rightarrow \T_{i}(S/R,M_3)\rightarrow \cdots.
		\]
		\item[\textup{(iii)}] Given a sequence of ring homomorphisms $R\rightarrow S\rightarrow T$ and a $T$-module $M$, there exists a long exact sequence
		\[
		\cdots \rightarrow \T_{i+1}(T/S,M) \rightarrow \T_{i}(S/R,M)\rightarrow \T_{i}(T/R,M) \rightarrow \T_{i}(T/S,M)\rightarrow \cdots.
		\]
	\end{enumerate}
\end{Proposition}
\begin{proof}
	We have $\T_i(S/R,M)=\H_i((\Omega_{X/R}\otimes_XS)\otimes_SM)$ for all $i$, and
	\[
	(\Omega_{X/R}\otimes_XS, \partial): \cdots \xrightarrow{\partial} \bigoplus_{\lambda\in \Lambda_1}SdT_{1\lambda}\xrightarrow{\partial}
	\bigoplus_{\lambda\in \Lambda}SdZ_{\lambda}\rightarrow 0.
	\]
	Since the derivation $d$
	commutes with $\partial$, it follows that for each $\lambda\in \Lambda_1$ we have
	\[
	\partial(dT_{1\lambda})=d(\partial T_{1\lambda})=d(a_\lambda)=
	\bigoplus_{\lambda'\in \Lambda}\frac{\partial a_{\lambda}}{\partial Z_{\lambda'}}dZ_{\lambda'} \quad\mod I.
	\]
	This shows that $\H_0(\Omega_{X/R}\otimes _RS)=\Omega_{S/R}$ and $\H_0(\Omega_{X/R}\otimes_XM)=\Omega_{S/R}\otimes_SM$.\smallskip
	
	(ii) Taking homology of the short exact sequence of complexes
	\[
	0\rightarrow \Omega_{X/R}\otimes_XM_1\rightarrow \Omega_{X/R}\otimes_XM_2\rightarrow \Omega_{X/R}\otimes_XM_3\rightarrow 0,
	\]
	we obtain the long exact sequence in (ii).\smallskip
	
	(iii) Let $X_{S/R}$, $X_{T/S}$ and $X_{T/R}$ be resolvents of $R\rightarrow S$, $S\rightarrow T$ and the composition $R\rightarrow S\rightarrow T$, respectively. It can be seen that if one makes the appropriate choices in the construction of these resolvents, we have a short exact sequence of complexes
	$$
	0\rightarrow\Omega_{X_{S/R}}\otimes T\rightarrow \Omega_{X_{T/R}}\otimes T\rightarrow\Omega_{X_{T/S}}\otimes T\rightarrow0.
	$$
	Tensoring this sequence  with the module $M$, and taking the long exact sequence induced by homology yields the assertion.
\end{proof}

\section{Cotangent functors of quotient rings}\label{sec3}

Now we will focus our attention to the following situation: $(R,\mm,K)$ is a Noetherian local ring, $I\subset R$ is a proper ideal, $S=R/I$ and $\varphi: R\rightarrow S$ is the canonical epimorphism. In this case the resolvent $X$ of $\varphi$ can be chosen such that $X_0=R$ and that $X_i$ is a free $R$-module for all $i\geq 0$. Since $X$ is acyclic with $\H_0(X)=S$, it follows that $X$ is a free $R$-resolution of $S$.

For a finitely generated $S$-module $N$, let $\mu(N)$ be the minimum cardinality of a set of generators of $N$. Notice that, by construction, $X^{(i)}$ is the subalgebra of $X$ generated by all elements of degree $\leq i$. For each $i$ we have $X^{(i)}=X^{(i-1)}[T_{i1}, \ldots, T_{ie_i}]$, where the elements $\partial T_{ij}$ are the cycles whose homology classes generate $\H_{i-1}(X^{(i-1)})$. We say that $X$ is a \textit{minimal resolvent} of $\varphi$ if $e_i=\mu(\H_{i-1}(X^{(i-1)}))$ for all $i\ge1$.

Let $\widetilde{\varphi}: X\rightarrow S$ be the augmentation homomorphism, and let $\widetilde{I}=\Ker(\widetilde{\varphi})$. Then $\widetilde{I}$ is an exact subcomplex of $X$ with $\widetilde{I}_0=I$ and $\widetilde{I}_i=X_i$ for $i\geq 1$. Since $\widetilde{I}$ is an ideal of $X$, we may consider $\widetilde{I}/\widetilde{I}^2$, which is a complex of $S$-modules with
\begin{equation}\label{eq:II2-homology}
	(\widetilde{I}/\widetilde{I}^2)_i= \begin{cases}
	\,I/I^2 & \text{if}\ i=0,\\
	(X^{(i)}/X^{(i-1)})_i\otimes_RS & \text{if}\ i>0.
\end{cases}
\end{equation}
Then $\widetilde{I}/\widetilde{I}^2$ is called the {\em cotangent complex} of $\varphi$ (with respect to the resolvent $X$).\smallskip

Note that for all $i>0$, $(\widetilde{I}/\widetilde{I}^2)_i=(X^{(i)}/X^{(i-1)})_i\otimes_RS=ST_{i1}+ST_{i2}+\cdots +ST_{ie_i}$ is free $S$-module, which is isomorphic to $\H_i(X^{(i)}/X^{(i-1)})$.
We have
\begin{Lemma}
	\label{natur}
	There exists a natural exact sequence of complexes
	\[
	0\rightarrow I/I^2\rightarrow \widetilde{I}/\widetilde{I}^2\rightarrow \Omega_{X/R}\otimes _XS\rightarrow 0
	\]
	of $S$-modules, where $I/I^2$ is considered to be a complex concentrated in degree $0$.
\end{Lemma}
\begin{proof}
	The derivation $d: \widetilde{I}\rightarrow \Omega_{X/R}\otimes_RS$ is a complex homomorphism, since $d$ commutes with $\partial$, and it vanishes on the subcomplex $\widetilde{I}^2$, because of the product rule for $d$. Therefore, $d$ induces a natural map of complexes $\alpha:\widetilde{I}/\widetilde{I}^2\rightarrow \Omega_{X/R}\otimes_RS$.
	
	Since $(\widetilde{I}/\widetilde{I}^2)_0=I/I^2$ and $(\Omega_{X/R}\otimes_RS)_0=0$, it follows that $\Ker(\alpha_0)=I/I^2$. On the other hand, for $i>0$, we have
	\[
	\alpha_i: (\widetilde{I}/\widetilde{I}^2)_i=\bigoplus_{j=1}^{e_i}ST_{ij}\rightarrow \bigoplus_{j=1}^{e_i}Sd T_{ij} =(\Omega_{X/R}\otimes_RS)_i, \quad T_{ij}\mapsto dT_{ij}.
	\]
	This shows that $\alpha_i$ is an isomorphism for $i>0$, and proves the lemma.
\end{proof}
Let $L$ be the cokernel of $I/I^2\rightarrow \widetilde{I}/\widetilde{I}^2$. Then $L_0=0$, $L_i=(\widetilde{I}/\widetilde{I}^2)_i$ for $i\geq 1$ and consequently $L\cong \Omega_{X/R}\otimes _XS$. Therefore, for any $S$-module $M$ we have
\begin{equation}\label{eq:LT}
	\H_i(L\otimes_SM)\cong \T_i(S/R,M)\quad \text{for} \quad i\geq 1.
\end{equation}
In particular, $\T_0(S/R,M)=0$ and $\T_1(S/R,M)=I/I^2\otimes_SM$.

\begin{Corollary}
	The homology of $\widetilde{I}/\widetilde{I}^2$ does not depend on the chosen resolvent $X$.
\end{Corollary}
\begin{proof}
	The short exact sequence given in Lemma \ref{natur} induces the long exact sequence
	\[
	\cdots\rightarrow\H_{i}(I/I^2)\rightarrow\H_i(\widetilde{I}/\widetilde{I}^2)\rightarrow\H_{i}(\Omega_{X/R}\otimes_XS)\rightarrow\H_{i-1}(I/I^2)\rightarrow\cdots.
	\]
	Since $I/I^2$ is a complex concentrated in degree zero, we then obtain that $\H_i(\widetilde{I}/\widetilde{I}^2)\cong\H_{i}(\Omega_{X/R}\otimes_XS)=T_i(S/R,S)$. Theorem \ref{tidef} yields the assertion.
\end{proof}



For all $i>0$, the short exact sequence
\begin{equation}\label{eq:seq1}
	0\rightarrow X^{(i-1)}\rightarrow X^{(i)} \rightarrow X^{(i)}/X^{(i-1)}\rightarrow 0
\end{equation}
induces the exact sequence
\begin{equation}\label{eq:seq2}
\H_i(X^{(i)})\xrightarrow{\eta_i} \H_i(X^{(i)}/X^{(i-1)})\xrightarrow{\sigma_{i-1}}\H_{i-1}(X^{(i-1)})\rightarrow 0,
\end{equation}
where $\sigma_{i-1}$ is the connecting homomorphism and the exactness on the right follows from the fact that $\H_{i-1}(X^{(i)})=\H_{i-1}(X)=0$.

\begin{Lemma}
	\label{tkernel}
	$\T_{i+1}(S/R,S)\cong \Ker(\eta_i)$ for all $i>0$.
\end{Lemma}
\begin{proof}
	For all $i>0$, the following diagram
	\begin{center}
		\[
		\xymatrix{
			\H_{i+1}(X^{(i+1)}/X^{(i)}) \ar[r]^{\quad\quad\,\sigma_{i}} \ar[d]_{\cong} & \H_{i}(X^{(i)}) \ar[r]^{\eta_{i}\quad\,\,} & \H_i(X^{(i+1)}/X^{(i)}) \ar[d]^{\cong} \\
			L_{i+1} \ar[rr]_{\partial_{i+1}} & & L_{i}
		}
		\]
	\end{center}
	is commutative. By equation (\ref{eq:LT}) we have
	\begin{equation}\label{eq:3}
		\begin{aligned}
			\T_i(S/R,S)\cong\H_i(L)\ &=\ \Ker(\partial_ {i+1})/\Im(\partial_{i+2})\cong \Ker(\eta_i\circ \sigma_i)/\Im(\eta_{i+1}\circ \sigma_{i+1})\\&\cong\ \sigma_i^{-1}
			(\Ker(\eta_i))/\Im(\eta_{i+1}),
		\end{aligned}
	\end{equation}
	where the last isomorphism follows from the fact that $\sigma_i$ is surjective.
	
	Consider the following commutative diagram
	\[
	\xymatrix{
		0 \ar[r] & \Im(\eta_{i+1}) \ar[r]^{\alpha\quad\,} \ar@{=}[d] & \sigma_i^{-1}(\Ker(\eta_i)) \ar[r]^{\quad\,\,\beta} \ar@{^{(}->}[d]^{} & \Ker(\eta_i) \ar[r] \ar@{^{(}->}[d]^{} & 0 \\
		0 \ar[r] & \Im(\eta_{i+1}) \ar[r] & \H_{i+1}(X^{(i+1)}/X^{(i)}) \ar[r]_{\quad\quad\sigma_i} & \H_i(X^{(i)}) \ar[r] & 0
	}
	\]
	
	The map $\beta$ is surjective because $\sigma_i$ is surjective by (\ref{eq:seq1}), and $\alpha$ is injective because the inclusion map $\Im(\eta_{i+1}) \rightarrow \H_{i+1}(X^{(i+1)}/X^{(i)})$ is injective. Let $x\in\Ker(\beta)$. Then $x\in \Ker(\sigma_i)$. Since the bottom row is exact, due to (\ref{eq:seq2}), it follows that $x\in\Im(\eta_{i+1})$. This shows that the top sequence is exact. It follows that $$\sigma_i^{-1}(\Ker(\eta_i))/\Im(\eta_{i+1})\cong\Ker(\eta_i),$$ which together with equation (\ref{eq:3}) implies the assertion.
\end{proof}

\begin{Corollary}
	\label{descriptiontwo}
	$\T_{i+1}(S/R,S)\cong \H_i(X^{(i-1)})$ for all $i\ge3$.
\end{Corollary}

\begin{proof}
	The short exact sequence (\ref{eq:seq1}) induces the long exact sequence
	\[
	\cdots \rightarrow \H_{i+1}(X^{(i)}/X^{(i-1)})\rightarrow \H_i(X^{(i-1)})\rightarrow \H_i(X^{(i)})\xrightarrow{\eta_i}\cdots.
	\]
	By Lemma \ref{tkernel} we have
	$\T_{i+1}(S/R,S)\cong \Ker(\eta_i)$ for all $i>0$. We claim that $\H_{i+1}(X^{(i)}/X^{(i-1)})=0$ for all $i\ge3$. Then the map $\psi:\H_i(X^{(i-1)})\rightarrow\H_i(X^{(i)})$ is injective and so $\T_{i+1}(S/R,S)\cong \Ker(\eta_i)=\Im(\psi)\cong\H_i(X^{(i-1)})$, as desired.
	
	It remains to show that $\H_{i+1}(X^{(i)}/X^{(i-1)})=0$ for all $i\ge3$. We even claim that $\H_j(X^{(i)}/X^{(i-1)})= 0$ for all $i<j<2i-1$. Indeed, since $X^{(i)}=X^{(i-1)}[T_{i1},\ldots, T_{ie_i}]$ we have $(X^{(i)}/X^{(i-1)})_j=
	\bigoplus_{j=1}^{e_i}(X^{(i-1)})_{j-i}T_{ij}$ for all integers $i<j< 2i-1$. Hence $\H_j(X^{(i)}/X^{(i-1)})= 0$ for all $i<j< 2i-1$, since $(X^{(i)})_{j-i}=0$ in this range.
\end{proof}

\section{Koszul homology, $\Tor(S,S)$ and the first Cotangent modules}\label{sec4}

The complex $\cdots \rightarrow \widetilde{I}_3 \rightarrow \widetilde{I}_2 \rightarrow \widetilde{I}_1\rightarrow 0$ is a $R$-free resolution of $I$. This implies that $\H_0(\widetilde{I}/I\widetilde{I})=0$ and $\H_i(\widetilde{I}/I\widetilde{I})\cong \Tor_{i-1}^R(S, I)\cong \Tor_{i}^R(S, S)$ for $i>0$.

\medskip
Corollary~\ref{descriptiontwo} together with the short exact sequence
\[
0\rightarrow \widetilde{I}^2/I\widetilde{I}\rightarrow \widetilde{I}/I\widetilde{I}\rightarrow \widetilde{I}/\widetilde{I}^2 \rightarrow 0
\]
yields the long exact sequence
\begin{eqnarray*}
\cdots \rightarrow \T_3(S/R,S)\rightarrow \H_2(\widetilde{I}^2/I\widetilde{I})\rightarrow \Tor_2^R(S,S)\rightarrow \T_2(S/R,S)\rightarrow 0\\
\cdots \rightarrow \H_i(\widetilde{I}^2/I\widetilde{I})\rightarrow \Tor_i^R(S,S)\rightarrow \T_i(S/R,S)\rightarrow \H_i(\widetilde{I}^2/I\widetilde{I})\to\cdots
\end{eqnarray*}
Here we also used that $\H_1(\widetilde{I}^2/I\widetilde{I})=0$, which is the case simply because $(\widetilde{I}^2/I\widetilde{I})_1=0$.

We can also compute $\H_2(\widetilde{I}^2/I\widetilde{I})$. Indeed, since $(\widetilde{I}^2/I\widetilde{I})_1=0$, then $\H_2(\widetilde{I}^2/I\widetilde{I})$ is the cokernel of the natural induced map
\[
(\widetilde{I}^2/I\widetilde{I})_3\xrightarrow{\overline{\partial}} (\widetilde{I}^2/I\widetilde{I})_2. 
\]

Note that 
\[
(\widetilde{I}^2/I\widetilde{I})_2=(IX_2+X_1^2)/IX_2\cong X_1^2/IX_2\cap X_1^2=X_1^2/IX_1^2=\bigoplus_{i<j}ST_{1i}
T_{1j}. 
\]
Similarly one shows that 
\[
(\widetilde{I}^2/I\widetilde{I})_3 \cong \bigoplus_{i,j}ST_{1i}T_{2j}\oplus \bigoplus_{i<j<k}ST_{1i}T_{1j}T_{1k}.
\]
Thus we obtain the exact sequence
\[
V_1\oplus V_2\xrightarrow{\overline{\partial}}\bigoplus_{i<j}ST_{1i}T_{1j}\rightarrow \H_2(\widetilde{I}^2/I\widetilde{I})\rightarrow 0,
\]
where $V_1=\bigoplus_{i,j}ST_{1i}T_{2j}$, $V_2=\bigoplus_{i<j<k}ST_{1i}T_{1j}T_{1k}$, and where $\overline{\partial}(V_2)=0$ and $\overline{\partial}(T_{1i}T_{2j})=T_{1i}\partial(T_{2j})\mod I$. From this it follows that
\[
\H_2(\widetilde{I}^2/I\widetilde{I})\cong \bigwedge^2I/I^2\cong \bigwedge^2 \Tor_2(S,S).
\]
As a consequence of this discussion we obtain
\begin{Proposition}
\label{wedget}	
There exists and exact sequence 
\[
\T_3(S/R,S)\rightarrow \bigwedge^2I/I^2\rightarrow \Tor_2^R(S,S)\rightarrow \T_2(S/R,S)\rightarrow 0.
\]
\end{Proposition}
Thus, together with Corollary~\ref{t3H} we have 

\begin{Corollary} \label{tor2}
Suppose that $\H_1^2=\H_2$ and that $\T_3(S/R,S)=0$. Then $$\Tor_2^R(S,S)\cong \bigwedge^2I/I^2.$$ 
\end{Corollary}

Let $X$ be a minimal resolvent of $I$. Then $X^{(1)}$ is the Koszul complex for a minimal set of generators of $I$. We will simply write $\H_i=\H_i(X^{(1)})$ for the Koszul homology.

\begin{Proposition}	\label{tkos}
There exists an exact sequence 	
\[
\H_3\rightarrow \T_4(S/R,S)\rightarrow \bigwedge^2\H_1\rightarrow \H_2\rightarrow \T_3(S/R,S)\rightarrow 0.
\]
\end{Proposition}

\begin{proof}
The short exact sequence $0\rightarrow X^{(1)}\rightarrow X^{(2)}\rightarrow X^{(2)}/X^{(1)}\rightarrow 0$ yields the long exact sequence
\[
\H_3\rightarrow \H_3(X^{(2)})\xrightarrow{\eta_3}\H_3(X^{(2)}/X^{(1)})\rightarrow \H_2\rightarrow \H_2(X^{(2)})\xrightarrow{\eta_2}\cdots
\]
By Corollary~\ref{descriptiontwo} we have the isomorphism $\H_3(X^{(2)})\cong T_4(S/R,S)$, and by Lemma~\ref{tkernel} we have $\T_3(S/R,S)\cong \Ker(\eta_2)=\Im(\H_2\rightarrow\H_2(X^{(2)}))$. Thus, we obtain the exact sequence
\[
\H_3\rightarrow T_4(S/R,S)\rightarrow \H_3(X^{(2)}/X^{(1)})\rightarrow \H_2\rightarrow \T_3(S/R,S)\rightarrow 0.
\]
It remains to be shown that $\H_3(X^{(2)}/X^{(1)})\cong \bigwedge^2 \H_1$. Indeed, $\H_3(X^{(2)}/X^{(1)})$ is the homology of the complex
\[
\bigoplus_j(X^{(1)})_2T_{2j}\oplus \bigoplus_{i<j}RT_{2i}T_{2j}\rightarrow \bigoplus_j(X^{(1)})_1T_{2j}\rightarrow \bigoplus_jRT_{2j}
\]
It follows that
$
\H_3(X^{(2)}/X^{(1)})\cong (\bigoplus_j \H_1T_j)/U,
$
where $U$ is the module generated by the elements $[\partial T_{2i}]T_j+T_{2i}[\partial T_{2j}]$ with $i<j$. This yields the desired conclusion.
\end{proof}

Since the image of $\bigwedge^2 \H_1$ in $\H_2$ is just $\H_1^2$, we obtain
\begin{Corollary}\label{t3H}
$\T_3(S/R,S)\cong \H_2/\H_1^2$.
\end{Corollary}

Let $R$ be a local ring and $I\subset R$ be an ideal. We say that $I$ is a \textit{complete intersection} if $\height(I)=\mu(I)$. If $\mu(I)=\height(I) +1$ we say that $I$ is an \textit{almost complete intersection}. If $R$ is Cohen-Macaulay, then $I$ is a complete intersection if and only $I$ is generated by a regular sequence. 

\begin{Corollary}\label{notvan}
Let $R$ be a regular local ring, and $I$ be a perfect ideal such that either 
\begin{enumerate}
\item[(i)] $\height(I)=2$ and I is generically a complete intersection, or
\item[(ii)] $I$ is generated by at most $\height(I)+1$ many elements. 
\end{enumerate}
Then $\T_4(S/R,S)=0$ if and only if $I$ is a complete intersection, where $S=R/I$.
\end{Corollary}

\begin{proof}
Since all $\T_i(S/R,S)=0$ if $I$ is a complete intersection, we only need to show that under the given conditions $\T_4(S/R,S)\neq 0$, if $I$ is not a complete intersection. 

(i) Suppose that $I$ is generated by $n+1$ elements. We may assume that $n>2$, otherwise we are in the case (ii). It is shown in \cite[Theorem (2.1), (g)(ii)]{AH} that $\Ker(\bigwedge^i \H_1\rightarrow \H_i)\neq 0$ for $i=2,\ldots,n-1$. Since $n>2$, then $\Ker(\bigwedge^2 \H_1\rightarrow \H_2)\neq 0$. By Proposition~\ref{tkos}, we have $$\Im(\T_4(S/R,S)\rightarrow\bigwedge^2 \H_1)=\Ker(\bigwedge^2 \H_1\rightarrow \H_2)\neq 0,$$ and this implies the assertion.

(ii) Since $I$ is not a complete intersection, we have $\mu(I)=\height(I)+1$. In other words, $I$ is an almost complete intersection. This implies that $\H_i=0$ for $i\geq 2$, and so by Proposition~\ref{tkos}, the map $\T_4(S/R,S)\rightarrow \bigwedge^2 \H_1$ is surjective. So, it suffices to show that $\bigwedge^2 \H_1\ne 0$. Note that $\H_1$ is isomorphic the canonical module $\omega_S$. By a theorem of Kunz \cite[Corollary 1.2]{Ku}, an almost complete intersection is never Gorenstein. Hence, $\mu(\H_1)>1$. Since $R$ is local, this implies that $\bigwedge^2 \H_1\neq 0$.
\end{proof}\bigskip

\section{Herzog's Last Theorem}\label{sec5}

Let $L=\Coker(\widetilde{I}/\widetilde{I}^2\rightarrow\Omega_{X/R}\otimes_XS)$ be the complex introduced in Section \ref{sec3}, where $R$ is a Noetherian local ring, $I\subset R$ is an ideal, $S=R/I$, the map $\varphi:R\rightarrow S$ is the canonical epimorphism and $X$ is a resolvent of $\varphi$.\smallskip

By equation (\ref{eq:LT}) we have $\H_i(L)\cong\T_i(S/R,S)$ for all $i\ge1$. Since in particular $\H_1(L)\cong\T_1(S/R,S)=I/I^2\otimes_SS\cong I/I^2$ we can regard the complex $L$ as an approximation of a resolution of $I/I^2$. In view of Vasconcelos Conjecture \ref{ConjB} on the conormal module (now a Theorem of Briggs \cite[Theorem A]{BB}) and several evidence, in 1981 J\"urgen Herzog posed the following conjecture.\medskip

\begin{Conjecture}\label{conj}
\textup{(Herzog's Last Theorem).} Let $R$ be a regular local ring, $I\subset R$ a proper ideal and $S=R/I$. Then the following conditions are equivalent:
\begin{enumerate}
\item[(a)] $I$ is a complete intersection.
\item[(b)] $\T_i(S/R, S)=0$ for all $i>1$.
\item[(c)] $\T_i(S/R, S)=0$ for all $i\gg1$.
\end{enumerate}
\end{Conjecture}

We would like to refer to this conjecture as \textit{Herzog's Last Theorem} both as a pun to the famous Fermat's Last Theorem and also because despite the fact that this question was posed in 1981, it is the last mathematical statement that was being studied by Professor Herzog at the end of his life.\smallskip

Conjecture \ref{conj} is widely open. Among other nice results Bernd Ulrich confirmed in \cite{Ul} the equivalence of (a) and (b) in the following cases: 
\begin{enumerate}
\item[(i)] (Theorem 2.20) $I$ belongs to the linkage class of a complete intersection.
\item[(ii)] (Proposition 2.38) $I$ is a perfect Gorenstein ideal of grade $3$ and deviation $2$.
\end{enumerate}
Theorem~2.20 contains a few other interesting equivalent conditions for $I$ to be a complete intersection. Note that Corollary~\ref{notvan}(i) is a special case of Theorem~2.20. Other results in \cite{Ul} only require that $R$ is Gorenstein. The reader is advised to consult this paper for details. 

It is obvious that (a) $\Rightarrow$ (b) $\Rightarrow$ (c). We conclude this section, by presenting a class of ideals for which (c) $\Rightarrow$ (a).\smallskip 

Let $(S, \mm,K)$ be a Noetherian local ring and $M$ be a finitely generated $S$-module. Then $\Tor_i(S/\mm,M)$ is a finitely generated $K$-vector space, and one calls the number $\beta_i(M)=\dim_K\Tor_i(S/\mm,M)$ the {\em $i$th Betti number} of $M$ and the formal series $P_M(t)=\sum_{i\geq 0}\beta_i(M)t^i$ the {\em Poincar\'{e} series} of $M$.

A formal series $P(t)\in\QQ[[t]]$ is called \textit{rational}, if it can be written as fraction $A(t)/B(t)$, where $A(t)$ and $B(t)$ are polynomials belonging to $\QQ[t]$.

For a longer time before the 80ties it was expected that all Poincar\'{e} series are rational. But in 1980, Anick \cite{Ak} found a counterexample. On the other hand, for a given ring in many cases $P_M(t)$ happens to be a rational function for all modules $M$ over this ring. For example this is the case for \textit{Golod} rings.\smallskip

For the proof of the last result, we need the following theorem of Mahler \cite{Ma}.
\begin{Theorem}\label{Thm:Ma}
	Let $A(t)=\sum_{i=0}^\infty\alpha_i\in\QQ[[t]]$ be a rational formal series. There exists an integer $r>0$ and a subset $\{r_1,\dots,r_g\}$ of $\{0,\dots,r-1\}$ such that for all $i\gg0$ we have $\alpha_i=0$ if and only if $i\equiv r_i\,(\textup{mod}\,r)$ for some $i\in\{1,\dots,g\}$.
\end{Theorem}

The following result due to Herzog \cite[Theorem (4.15)]{H81} supports Conjecture \ref{conj}.
\begin{Theorem}
\label{true}
Let $R$ be a regular local ring, $I\subset R$ a proper ideal and $S=R/I$. Assume that all finitely generated $S$-modules have a rational Poincar\'{e} series. Then for the ideal $I$, Conjecture~\ref{conj} is correct. 
\end{Theorem}

\begin{proof}
If $I$ is a complete intersection, then the Koszul complex on a minimal set of generators of $I$ is the resolvent of $R\rightarrow S$. Therefore $\T_i(S/R,S)=0$ for all $i>1$. This proves that (a) $\Rightarrow$ (b). The implication (b) $\Rightarrow$ (c) is trivial.

For the proof of (c) $\Rightarrow$ (a), we choose a minimal resolvent $X$ of the map $R\rightarrow S$. Let $L=\Coker(\widetilde{I}/\widetilde{I}^2\rightarrow\Omega_{X/R}\otimes_XS)$. Then $\partial(L)\subseteq\mm L$, and this implies that
$$
\rank_S L_i=\dim_K\H_i(L\otimes K)=\dim_K \T_i(S/R,K).
$$

The sequence of ring homomorphisms $R\rightarrow S\rightarrow K$ induces the long exact Zariski sequence
\[
\cdots \rightarrow \T_i(S/R,K)\rightarrow \T_i(K/R,K)\rightarrow \T_i(K/S,K)\rightarrow \cdots.
\]
Since $R$ is regular, the maximal ideal $\mm$ is generated by a regular sequence, $R/\mm\cong K$ and so we have $\T_i(K/R,K)=0$ for all $i>0$. Therefore, 
\[
\T_i(S/R,K)\cong \T_{i+1}(K/S,K) \quad \text{for all} \quad i\geq 1.
\]
To compute $\T_{i+1}(K/S,K)$ we choose a minimal resolvent $Y$ for the residue class map $S\rightarrow K$. By a theorem of Gulliksen \cite[Page 55]{Gu}, $Y$ is a minimal free $S$-resolution of $K$. For each $i\geq 1$, we set $\epsilon_i=\dim_K(Y^{(i)}/Y^{(i-1)})_{i}\otimes_SK$. The integer $\epsilon_i$ is called the {\em $i$th deviation} of $S$. Our discussion together with formula (\ref{eq:II2-homology}) show that 
\[
\rank_S L_{i}=\epsilon_{i}
\quad \text{for}\quad i\geq 1.
\]
By assumption, there exists an integer $i_0>1$ such that $\T_i(S/R,S)=0$ for $i\geq i_0$. Let $M=\Coker(L_{i_0}\rightarrow L_{i_0-1})$, then 
\[
\cdots \rightarrow L_{i_0}\rightarrow L_{i_0-1}\rightarrow M\rightarrow 0
\]
is a minimal free $S$-resolution of $M$. Thus, $P_M(t)=\sum_{i\geq i_0}\epsilon_it^i$.

We may assume that $\varepsilon_i>0$ for all $i\ge i_0$. Indeed, assume that $\varepsilon_i=0$ for some $i\ge i_0$. Then $\varepsilon_j=0$ for all $j\ge i$. Indeed, in this case,
$$
0\rightarrow L_{i-1}\rightarrow\cdots\rightarrow L_{i_0}\rightarrow L_{i_0-1}\rightarrow M\rightarrow 0
$$
would be a minimal free $S$-resolution of $M$. But then the vanishing of the $\varepsilon_j$ for all $j\ge i$ would imply that $I$ is a complete intersection by a theorem of Gulliksen \cite{Gu1}.\smallskip

Hence, for the rest of the proof we assume that $\varepsilon_i>0$ for all $i\ge i_0$.

Let $F(t)=\sum_{i\geq 1}\epsilon_it^i$. By assumption, $P_M(t)$ is a rational, and hence $F(t)$ is rational as well. Also by our assumption, $P_K(t)$ is rational. Assume that $I$ is not a complete intersection. We will show that in this case $F(t)$ and $P_K(t)$ cannot be rational at the same time. This will lead to the desired contradiction.

Since $Y$ is a minimal resolvent of $S\rightarrow K$, it follows from \cite[Remark 7.1.1]{Av} (or also \cite[Corollary 1]{Gu}) that we can write
$$
P_K(t)=\prod_{i=1}^\infty\frac{(1+t^{2i-1})^{\varepsilon_{2i-1}}}{(1-t^{2i})^{\varepsilon_{2i}}}.
$$

We denote by $\log$ the \textit{logarithmic derivative}. That is, $\log g=g'/g$ for any $g\in\QQ[t]$. For polynomials $f,g\in\QQ[t]$ and $a\in\QQ$, the following three elementary rules hold:
\begin{enumerate}
	\item[(i)] $\log(fg)=\log(f)+\log(g)$,
	\item[(ii)] $\log(\frac{f}{g})=\log(f)-\log(g)$,
	\item[(iii)] $\log(f^a)=a\log(f)$.
\end{enumerate}

Using the properties (i), (ii) and (iii) we see that
\begin{align*}
	\log P_k(-t)\ &=\ \sum_{i=1}^\infty[\varepsilon_{2i-1}\log(1+(-1)^{2i-1}t^{2i-1})-\varepsilon_{2i}\log(1-(-1)^{2i}t^{2i})]\\
	&=\ \sum_{i=1}^\infty[-\frac{(2i-1)\varepsilon_{2i-1}t^{2i-2}}{1-t^{2i-1}}+\frac{(2i)\varepsilon_{2i}t^{2i}}{1-t^{2i}}]\\
	&=\ -\frac{\varepsilon_{1}}{1-t}+\frac{2\varepsilon_2t}{1-t^2}-\frac{3\varepsilon_{3}t^2}{1-t^3}+\cdots\\
	&=\ \sum_{i=0}^\infty(-1)^{i+1}(i+1)\varepsilon_{i+1}t^i\frac{1}{1-t^{i+1}}.
\end{align*}

Now, we use that $1/(1-t^{i+1})=\sum_{j=0}^\infty t^{(i+1)j}$ to obtain
\begin{align*}
	\log P_k(-t)\ &=\ \sum_{i=1}^\infty(-1)^{i+1}(i+1)\varepsilon_{i+1}t^i(\sum_{j=0}^\infty t^{(i+1)j})\\
	&=\ -\varepsilon_{1}+(-\varepsilon_{1}+2\varepsilon_2)t+(-\varepsilon_{1}-3\varepsilon_{3})t^2+(-\varepsilon_{1}+2\varepsilon_2+4\varepsilon_4)t^3+\cdots\\
	&=\ \sum_{i=0}^\infty(\sum_{j\ \textup{divides}\ i+1}(-1)^jj\varepsilon_j)t^i.
\end{align*}

Since $P_k(t)$ is rational by assumption, it follows that $\log P_k(-t)$ is also rational. Therefore, the power series $A(t)=(\log P_k(-t)+F'(t)+\sum_{i=0}^\infty\varepsilon_1t^i)t$ is again rational, and if we put $A(t)=\sum_{i=0}^\infty\alpha_it^i$, we see immediately that
$$
\alpha_i\ =\ \sum_{\substack{j\ \textup{divides}\ i\\ j\ne1,j\ne i}}(-1)^jj\varepsilon_j\quad\textup{for all odd integers}\ i.
$$

Since $\varepsilon_i>0$ for all $i\ge i_0$, it follows that for all odd integers $i\ge i_0$ we have $\alpha_i=0$ whenever $i$ is a prime number, and $\alpha_i<0$ if $i$ is not a prime, because the integers dividing an odd number are again odd. Since $A(t)$ is rational, this strange behavior of the $\alpha_i$ contradicts the Theorem \ref{Thm:Ma} of Mahler. The conclusion follows.
\end{proof}

\medskip
\noindent\textbf{Acknowledgment.}
A. Ficarra was partly supported by the Grant JDC2023-051705-I funded by
MICIU/AEI/10.13039/501100011033 and by the FSE+.

\end{document}